\documentclass[12pt]{amsart}
\usepackage{amsmath,amssymb,amsfonts}
\usepackage{graphicx,xcolor}
\usepackage{epstopdf}
\usepackage[colorlinks=true]{hyperref}
\hypersetup{urlcolor=blue, citecolor=red}

\usepackage[all,cmtip]{xy}

\usepackage{rotating}

\newtheorem{theorem}{Theorem}[section]

\newtheorem{lemma}[theorem]{Lemma}
\newtheorem{proposition}{Proposition}

\newtheorem*{problem}{Problem}
\theoremstyle{definition}
\newtheorem{definition}[theorem]{Definition}
\newtheorem{remark}{Remark}

\title[The coordinate functions of the Heighway dragon curve] %Use the shortened version of the full title
{The coordinate functions of \\ the Heighway dragon curve}

\keywords{Heighway dragon, paper-folding, box-counting dimension, coordinate functions.}

\begin{document}
	\maketitle

	{\footnotesize
		\centerline{\author{D. A. Caprio} \footnote{UNESP -  Departamento de matem\'atica, Faculdade de Engenharia - C\^ampus de Ilha Solteira, Avenida Brasil, 56 - Centro - Ilha Solteira, SP - CEP 15385-000, SP, Brasil. e-mail: {\rm \texttt{danilo.caprio@unesp.br}}}}
	}

\begin{abstract}
	In this work, we study properties of the coordinate functions $x_\theta$ and $y_\theta$ of the dragon curve associated with the angle $\frac{\pi}{3} < \theta < \frac{5\pi}{3}$, and we prove that the box-counting dimension of its graph is equal to $1 - \frac{\log \cos\alpha}{\log 2}$, where $\alpha = \frac{\pi - \theta}{2}$.
\end{abstract}
		
\section{Introduction}

The dragon curve (also known as the Heighway dragon or paper-folding curve) is a self-similar fractal that can be approximated through recursive methods such as Lindenmayer systems. It can also be constructed either by folding paper or through the following geometric procedure: beginning with a straight line segment and fixing the angle $\alpha=\frac{\pi}{4}$, we first replace this single segment with two segments, each scaled by a ratio $r-\frac{\sqrt{2}}{2}$ and arranged so the original segment would form the hypotenuse of an isosceles right triangle. In the next step, we replace each of these new segments with two segments forming a right angle, alternately rotated by $\alpha$ to the right and left, with each new segment again scaled by $r$. Continuing this process iteratively by replacing every existing segment with two new segments rotated by $\alpha$ (alternating directions) and scaled by $r$, we obtain in the limit the dragon curve associated with the angle $\theta= \pi-2\alpha = \frac{\pi}{2}$.

The dragon curve is a recursive non-intersecting curve and has been exhaustively studied by various mathematicians (see \cite{AKW}, \cite{N} and \cite{T}).

In this work, we consider a generalized class of dragon curves $\mathcal{D}_\theta$ parameterized by an angle $\theta \in \left(\frac{\pi}{3}, \frac{5\pi}{3}\right)$. The construction begins with $\alpha = \frac{\pi - \theta}{2}$ and an initial line segment $\mathcal{D}_{\theta,0}$. 

The recursive construction proceeds as follows: first, replace the initial segment with two new segments scaled by $r = (2\cos\alpha)^{-1}$, forming angle $\theta$ between them, which generates $\mathcal{D}_{\theta,1}$. Then, for each subsequent generation $\mathcal{D}_{\theta,n+1}$, every segment in $\mathcal{D}_{\theta,n}$ is replaced by two new segments scaled by $r$, maintaining angle $\theta$ but with alternating rotations by $\alpha$ (first right, then left, and so on). 

This process yields $\mathcal{D}_{\theta,1}$ with 2 segments, $\mathcal{D}_{\theta,2}$ with 4 segments, and generally $\mathcal{D}_{\theta,n}$ with $2^n$ segments. The dragon curve $\mathcal{D}_\theta$ emerges as the limit of this construction (Figure \ref{dragon} illustrates several examples of these curves for different angles $\theta$):
\begin{equation*}
	\mathcal{D}_\theta = \lim_{n\to\infty} \mathcal{D}_{\theta,n}.
\end{equation*}

\begin{figure}[!h]
	\centering
	\includegraphics[scale=0.245]{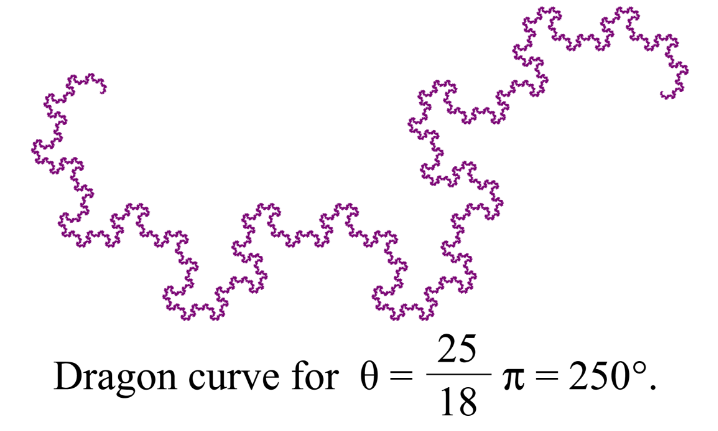}
	\includegraphics[scale=0.245]{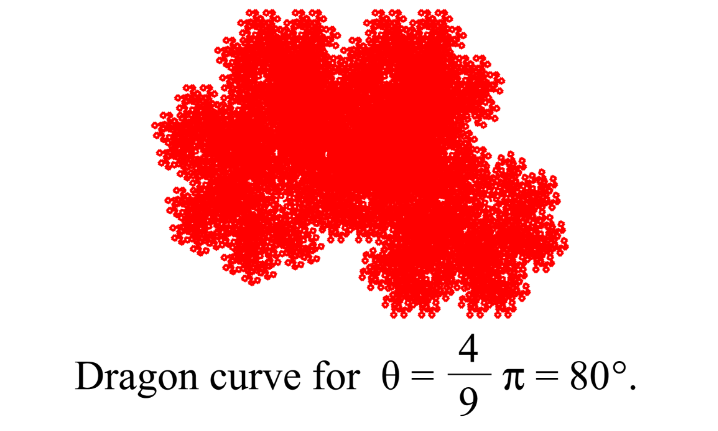}
	\includegraphics[scale=0.245]{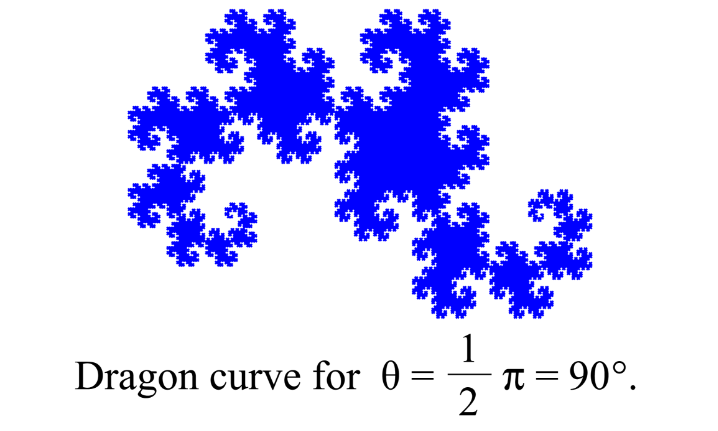} 
	\includegraphics[scale=0.245]{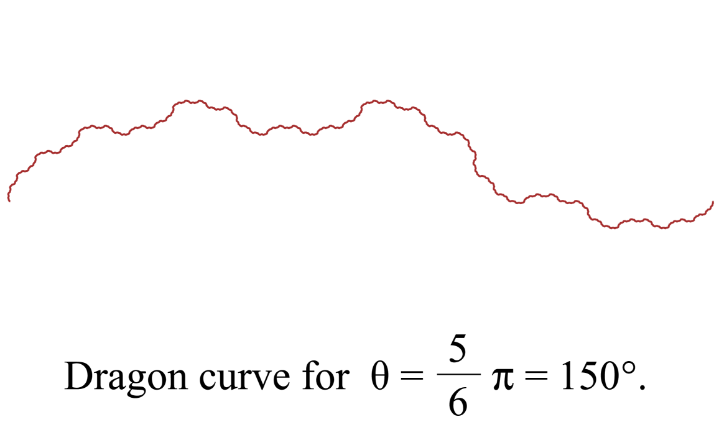}
	\caption{Examples of dragon curve.} \label{dragon}
\end{figure}

In this work, we analyze the parametrization $(x_\theta(t), y_\theta(t))$ of the dragon curve $\mathcal{D}_\theta$ for $t \in [0,1]$. Our main focus concerns the fractal properties of the coordinate functions, particularly their box-counting dimensions (Minkowski-Bouligand dimensions). 

We establish that for the graphs
\[
X_\theta := \{(t, x_\theta(t)) : t \in [0,1]\} \quad \text{and} \quad Y_\theta := \{(t, y_\theta(t)) : t \in [0,1]\},
\]
their box-counting dimensions coincide and are given by the exact formula:
\[
\dim_B X_\theta = \dim_B Y_\theta = 1 - \frac{\log(\cos\alpha)}{\log 2},
\]
where $\alpha = \frac{\pi - \theta}{2}$ is the associated rotation parameter.

The paper is organized as follows. Section~\ref{coordinate} presents the coordinate functions for each stage of the dragon curve's construction and demonstrates the convergence of these step functions to the final coordinate functions of $\mathcal{D}_\theta$, culminating in the statement of our main result. 
Section~\ref{lemmas} develops the necessary technical machinery, establishing all auxiliary results required for the proof of the main theorem. The complete proof is then presented in Section~\ref{proofmaintheorem}.
We conclude in Section~\ref{conclusion} by discussing several open problems and potential research directions that naturally emerge from this work.

\section{Coordinate functions of the dragon curve $\mathcal{D}_\theta$} \label{coordinate}

By our construction of $\mathcal{D}_\theta$, the angle parameter must satisfy $\theta \in \left(\frac{\pi}{3}, \frac{5\pi}{3}\right)$. Without loss of generality, we may restrict our analysis to the case $\theta \in \left(\frac{\pi}{3}, \pi\right)$, since for angles $\theta \in \left(\pi, \frac{5\pi}{3}\right)$, the resulting dragon curve is equivalent to the curve for angle $2\pi - \theta$ reflected across the horizontal axis. The special case $\theta = \pi$ yields a degenerate dragon curve that reduces to a straight line segment.

Fix $\theta \in \left(\frac{\pi}{3}, \pi\right)$. For each construction step $n \in \mathbb{N}$ of the dragon curve $\mathcal{D}_{\theta,n} = \mathcal{D}_n$, we consider the coordinate functions $x_n = x_{\theta,n} : [0,1] \longrightarrow \mathbb{R}$  and $y_n = y_{\theta,n} : [0,1] \longrightarrow \mathbb{R}$,
where $(x_n(t), y_n(t))$ parametrizes $\mathcal{D}_n$ for $t \in [0,1]$. 

The initial configuration $\mathcal{D}_0$ is the line segment connecting $(0,0)$ to $(1,0)$. For each $n \geq 0$, the curve $\mathcal{D}_n$ is parameterized by $(x_n(t), y_n(t))$ with $t \in [0,1]$, satisfying the boundary conditions $(x_n(0), y_n(0)) = (0,0)$ at the starting point and $(x_n(1), y_n(1)) = (1,0)$ at the endpoint.
Figure~\ref{dragoncurve} illustrates the first six iterations of this construction for $\theta = \frac{\pi}{2}$, with each subsequent step shown in red.

\begin{figure}[!h]
	\centering
	\includegraphics[scale=0.21]{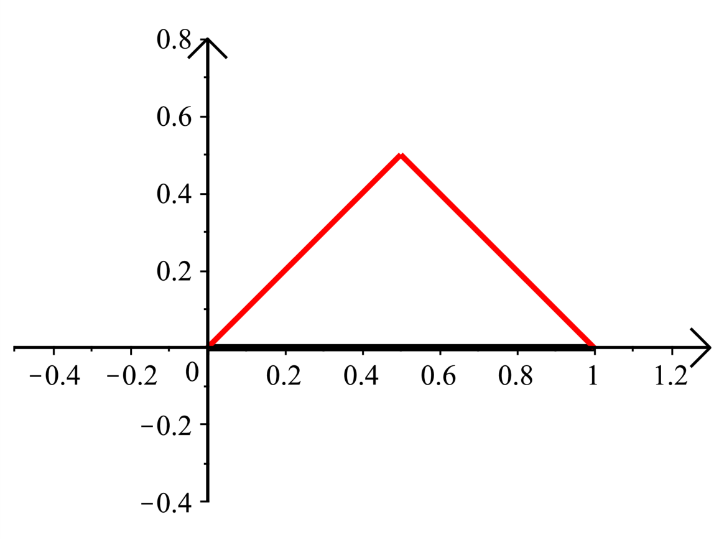}
	\includegraphics[scale=0.21]{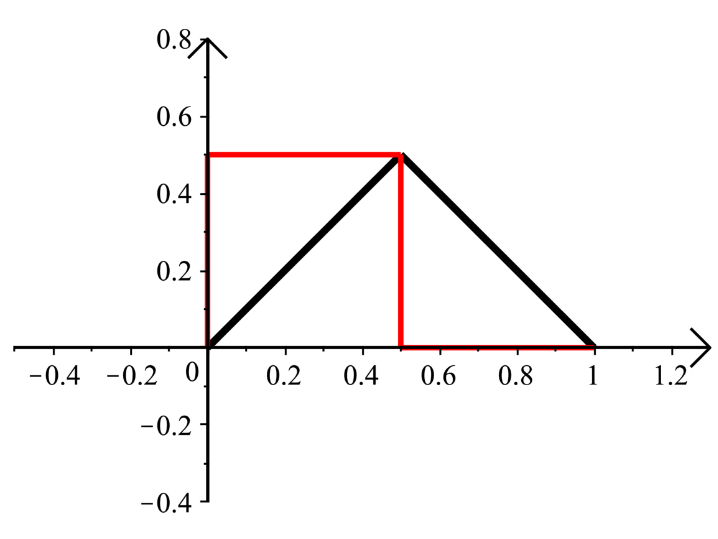}
	\includegraphics[scale=0.21]{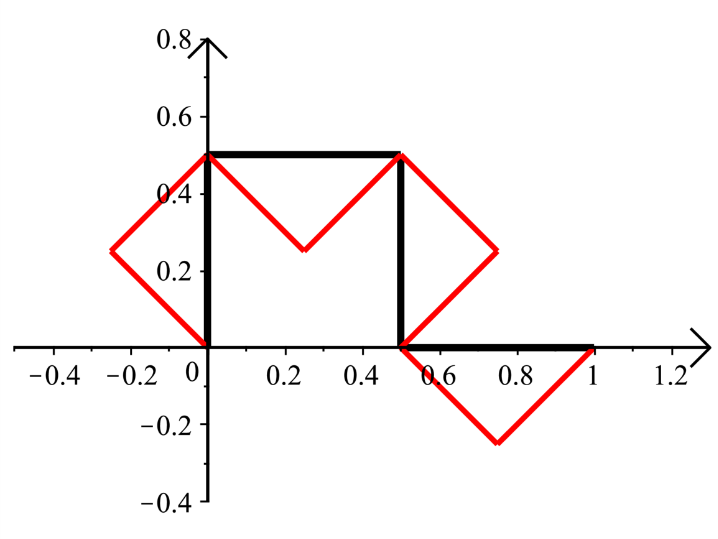}
	\includegraphics[scale=0.21]{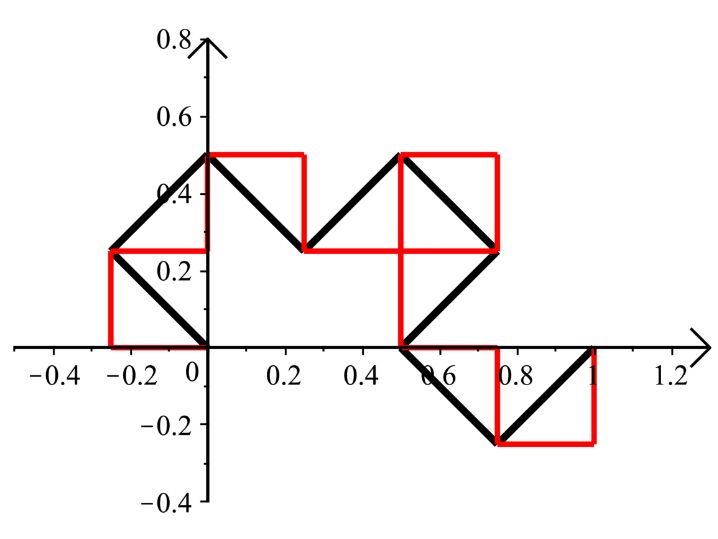}
	\includegraphics[scale=0.21]{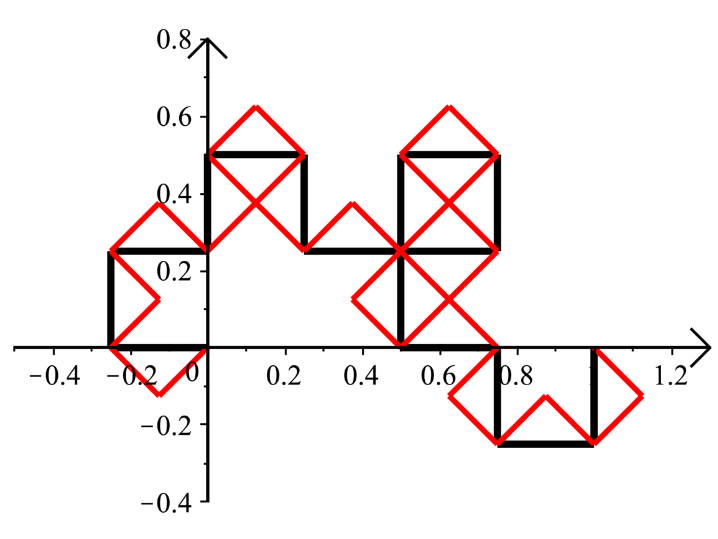}
	\includegraphics[scale=0.21]{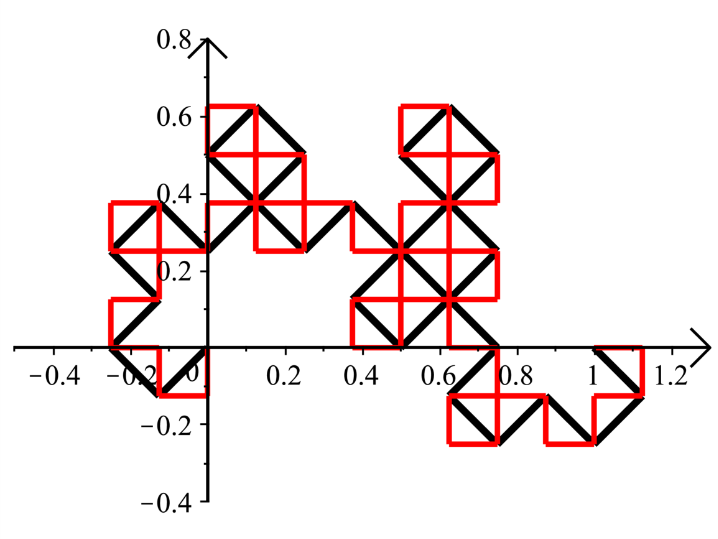}
	\caption{First steps in the dragon curve construction for $\theta = \frac{\pi}{2}$.} \label{dragoncurve}
\end{figure}

Figures~\ref{cordxk} and~\ref{cordyk} display graphical representations of the sets
\[
X_{\theta,k} := \{(t, x_k(t)) : 0 \leq t \leq 1\} \quad \text{and} \quad Y_{\theta,k} := \{(t, y_k(t)) : 0 \leq t \leq 1\},
\]
for iterations $k = 1, 2, 3$, respectively.

\begin{figure}[!h]
	\centering
	\includegraphics[scale=0.29]{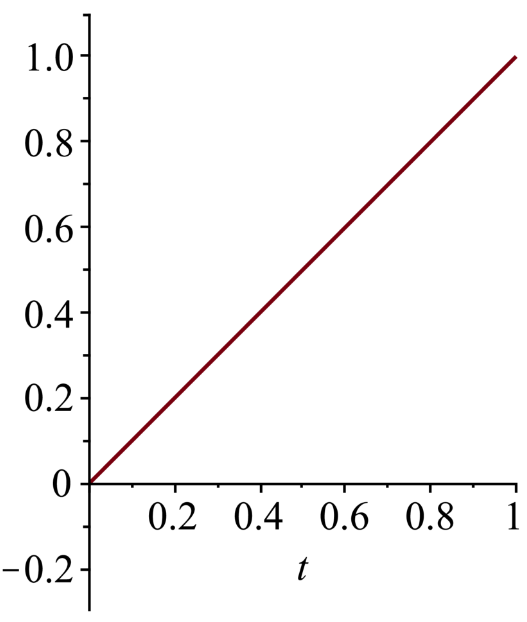}   
	\includegraphics[scale=0.29]{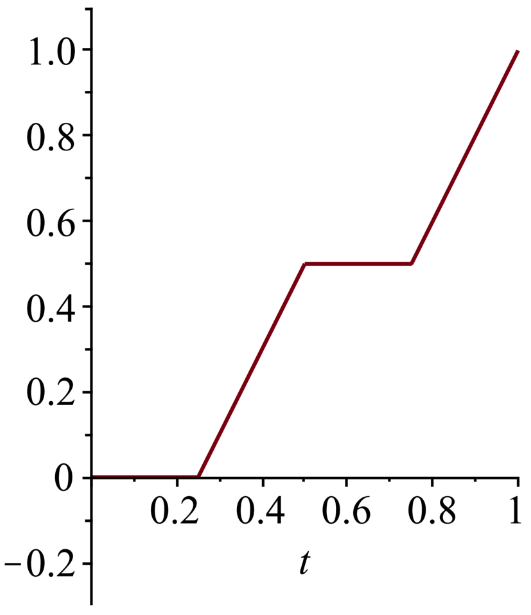}
	\includegraphics[scale=0.29]{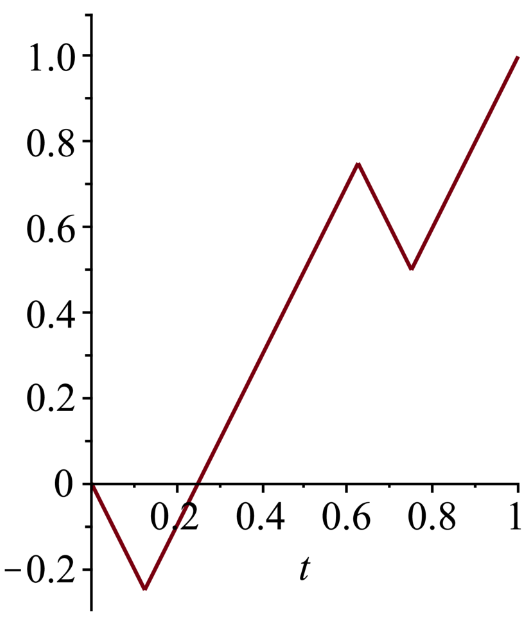}    
	\caption{Graphs of $x_{\frac{\pi}{2},1}$, $x_{\frac{\pi}{2},2}$ and $x_{\frac{\pi}{2},3}$, respectively.} \label{cordxk}
\end{figure}

\begin{figure}[!h]
	\centering
	\includegraphics[scale=0.24]{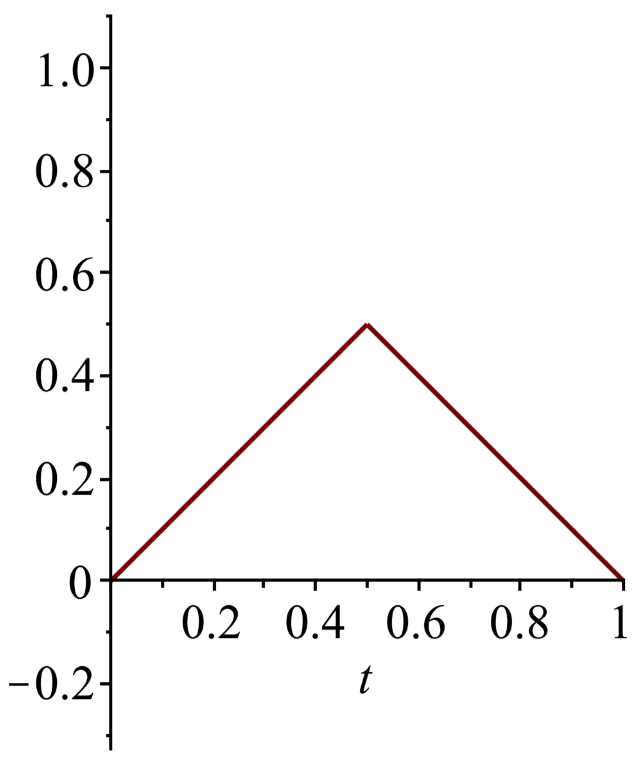}   
	\includegraphics[scale=0.24]{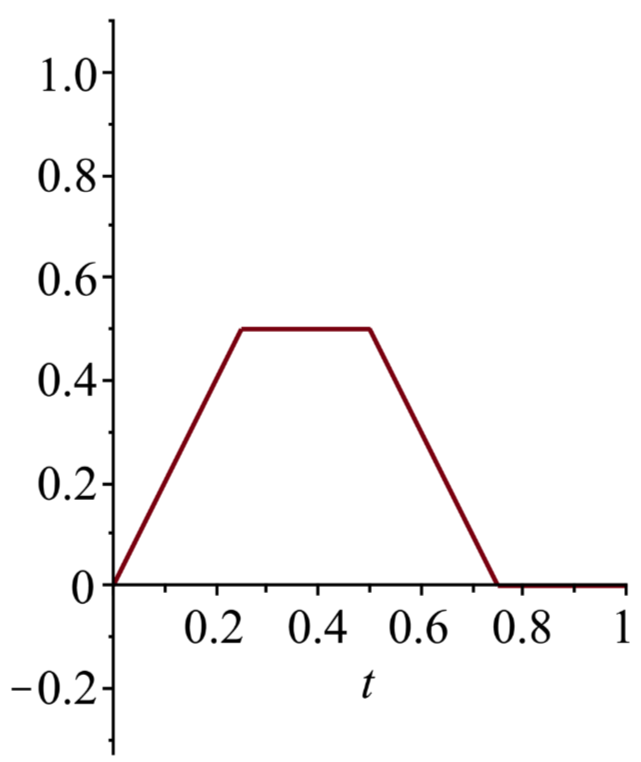}
	\includegraphics[scale=0.24]{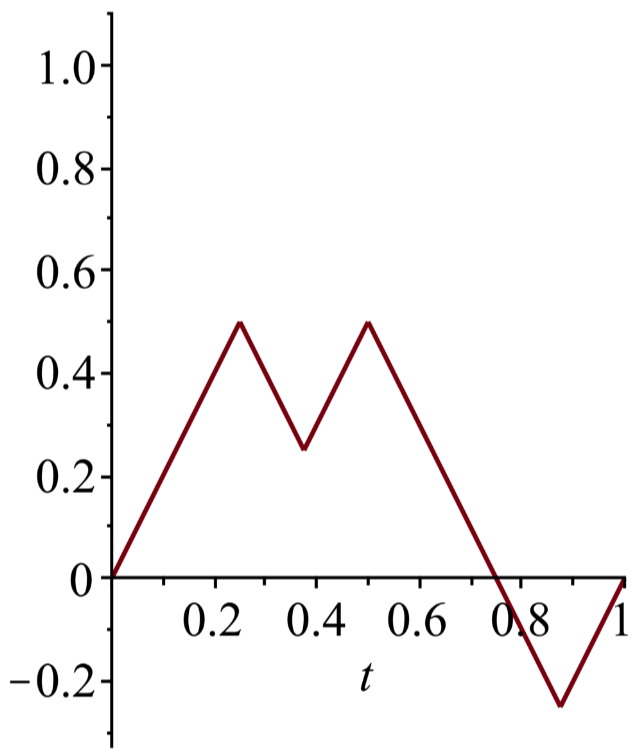}   
	\caption{Graphs of $y_{\frac{\pi}{2},1}$, $y_{\frac{\pi}{2},2}$ and $y_{\frac{\pi}{2},3}$, respectively.} \label{cordyk}
\end{figure}

In order to define the coordinate functions $x_n$ and $y_n$ of $\mathcal{D}_n$ for each $n \geq 0$, we first observe that by construction, each $\mathcal{D}_n$ consists of $2^n$ line segments, each of length $L_n = (2\cos\alpha)^{-n}$. 

Let $B := \{n\alpha : n \geq 0\}$ be the set of angle multiples, and let $B^*$ denote the set of all finite words over $B$. We define the substitution map $\psi: B^* \to B^*$ through concatenation:
\[
\psi((a_i)_{i=1}^k) = \psi(a_1)\cdots\psi(a_k),
\]
where for each letter $a_i$, the substitution rule is given by
\[
\psi(a_i) = \begin{cases}
	(a_i + \alpha), (a_i - \alpha) & \text{if } i \text{ is odd}, \\
	(a_i - \alpha), (a_i + \alpha) & \text{if } i \text{ is even}.
\end{cases}
\]

Therefore, at each construction step $k \in \mathbb{N}$, the angles of all line segments in the dragon curve $\mathcal{D}_k$ are encoded by a word $b_k \in B^*$. This sequence of words is defined recursively by:
\[
b_0 = 0 \quad \text{and} \quad b_k = \psi^k(b_0) = b_{k,1}\cdots b_{k,2^k} \quad \text{for } k \geq 1,
\]
where $\psi^k$ denotes the $k$-fold application of the substitution map $\psi$.

\begin{proposition}
	The coordinate functions $x_{k},y_{k}:[0,1]\longrightarrow \mathbb{R}$ of $\mathcal{D}_k$ are defined by $x_{k}(t)=2^k f_{k,1}t$ and $y_{k}(t)=2^k g_{k,1}t$, if $t\in \left[0, \frac{1}{2^k}\right]$,

\medskip 
	
	\begin{center} 
		$x_{k}(t)= 2^kf_{k,j}\left(t-\frac{j-1}{2^k}\right)+f_{k,1}+\ldots+f_{k,j-1}$
	\end{center} 

\medskip 
	
\noindent 	and

\medskip 

	\begin{center} 
		$y_{k}(t)= 2^kg_{k,j}\left(t-\frac{j-1}{2^k}\right)+g_{k,1}+\ldots+g_{k,j-1}$
	\end{center} 

\medskip 

\noindent 	if  $t\in \left[\frac{j-1}{2^k}, \frac{j}{2^k}\right]$ and $2\leq j\leq 2^k$,  where 

\medskip 

	\begin{center}
		$f_{k,i}=L_k\cos b_{k,i}=\frac{\cos b_{k,i}}{(2 \cos \alpha)^k}$ and $g_{k,i}=L_k\sin b_{k,i}=\frac{\sin b_{k,i}}{(2 \cos \alpha)^k}$,
	\end{center}

\medskip 

\noindent	for all $k\geq 0 $.
\end{proposition}
\begin{proof} For each $k\geq 0$, the projections of the $i-$th line segment of $\mathcal{D}_{k}$ in the $x$-axis and $y$-axis have lengths $|f_{k,i}|$ and $|g_{k,i}|$, respectively, where 
	\begin{center}
		$f_{k,i}=L_k\cos b_{k,i}=\frac{\cos b_{k,i}}{(2 \cos \alpha)^k}$ and $g_{k,i}=L_k\sin b_{k,i}=\frac{\sin b_{k,i}}{(2 \cos \alpha)^k}$.
	\end{center}
	
	From this, we have that the coordinate function $x_k:[0,1]\longrightarrow \mathbb{R}$ of the dragon curve in the step $k$, satisfies
	$$
	x_k:\left\{
	\begin{array}{ccl}
		0 & \longmapsto & 0, \\
		\frac{1}{2^k} & \longmapsto & f_{k,1}, \\
		& \vdots & \\
		\frac{j}{2^k} & \longmapsto & f_{k,1}+\ldots+f_{k,j-1} +f_{k,j},\\
		& \vdots & \\
		1 & \longmapsto & f_{k,1}+\ldots+f_{k,2^k} = 1,
	\end{array} \right.
	$$
	and the coordinate function $y_k:[0,1]\longrightarrow \mathbb{R}$, satisfies
	$$
	y_k:\left\{
	\begin{array}{ccl}
		0 & \longmapsto & 0, \\
		\frac{1}{2^k} & \longmapsto & g_{k,1}, \\
		& \vdots & \\
		\frac{j}{2^k} & \longmapsto & g_{k,1}+\ldots+g_{k,j-1} +g_{k,j},\\
		& \vdots & \\
		1 & \longmapsto & g_{k,1}+\ldots+g_{k,2^k} = 0.
	\end{array} \right.
	$$
	
	Therefore, in the step $k$, the coordinate function $x_k$ of $\mathcal{D}_k$ is defined by $x_k(t)=2^k f_{k,1}t$, if $t\in \left[0, \frac{1}{2^k}\right]$ and
	
	\begin{center} 
		$x_k(t)= 2^kf_{k,j}\left(t-\frac{j-1}{2^k}\right)+f_{k,1}+\ldots+f_{k,j-1}, \textrm{ if }  t\in \left[\frac{j-1}{2^k}, \frac{j}{2^k}\right],$
	\end{center} 
	
	\noindent for $2\leq j\leq 2^k$. Analogously, the coordinate function $y_k$ of $\mathcal{D}_k$ is defined by $y_k(t)=2^k g_{k,1}t$, if $t\in \left[0, \frac{1}{2^k}\right]$ and
	
	\begin{center} 
		$y_k(t)= 2^kg_{k,j}\left(t-\frac{j-1}{2^k}\right)+g_{k,1}+\ldots+g_{k,j-1}, \textrm{ if }  t\in \left[\frac{j-1}{2^k}, \frac{j}{2^k}\right],$
	\end{center} 
	
	\noindent for $2\leq j\leq 2^k$.
\end{proof}

\begin{lemma} \label{cauchy}
	The sequences of coordinate functions $(x_n)_{n\geq 0}$ and $(y_n)_{n\geq 0}$ are uniformly convergent in $\mathcal{C}[0,1]=\{f:[0,1]\longrightarrow \mathbb{R}: f \textrm{ is continuous}\}$.
\end{lemma}
\begin{proof}
	Since $\mathcal{C}[0,1]$ is a complete space, it is sufficient prove that $(x_n)_n$ and $(y_n)_n$ are Cauchy sequences in $\mathcal{C}[0,1]$.
	
	Let $n\geq 1$ and $j\in\{1,\ldots , 2^{n}\}$ and $t\in  \left[\frac{j-1}{2^{n}}, \frac{j}{2^{n}}\right]= \left[\frac{2j-2}{2^{n+1}}, \frac{2j}{2^{n+1}}\right]$. Hence
	
	\bigskip 
	
	\noindent $|x_{n+1}(t)-x_n(t)|=$ 
	
	\medskip 
	
	$|2^{n+1}f_{n+1,2j}\left(t-\frac{2j-1}{2^{n+1}}\right)+f_{n+1,1}+\ldots+f_{n+1,2j-1}$
	
	\medskip 
	
	\hfill $-2^{n}f_{n,j}\left(t-\frac{j-1}{2^{n}}\right)-f_{n,1}-\ldots-f_{n,j-1}|.$
	
	\bigskip 
	
	Since for $i\in\{1,\ldots,j-1\}$ we have
	\begin{center}
		$b_{n+1,2i-1}=b_{n,i}+\alpha$ and $b_{n+1,2i}=b_{n,i}-\alpha$ \\ or \\ $b_{n+1,2i-1}=b_{n,i}-\alpha$ and $b_{n+1,2i}=b_{n,i}+\alpha$, \end{center}
	it follows that
	\begin{center} 
		$f_{n+1,2i-1}+ f_{n+1,2i}=\frac{2\cos b_{n,i}\cos \alpha}{(2\cos \alpha)^{n+1}}=f_{n,i}.$
	\end{center} 
	Thus, 
	
	\medskip 
	
	\noindent $|x_{n+1}(t)-x_n(t)|= $ 
	
	\bigskip 
	
	\hspace{.1cm} $\begin{array}{l}
		|2^{n+1}f_{n+1,2j}\left(t-\frac{2j-1}{2^{n+1}}\right)+f_{n+1,2j-1}-2^{n}f_{n,j}\left(t-\frac{j-1}{2^{n}}\right)| \leq \\
		\\
		2^{n+1}|f_{n+1,2j}|\cdot \left|t-\frac{2j-1}{2^{n+1}}\right|+|f_{n+1,2j-1}|+2^{n}|f_{n,j}|\cdot \left|t-\frac{j-1}{2^{n}}\right|\leq  \\
		\\
		2^{n+1}\frac{1}{|2\cos \alpha|^{n+1}}\left|\frac{2j}{2^{n+1}}-\frac{2j-1}{2^{n+1}}\right|+\frac{1}{|2\cos \alpha|^{n+1}}+2^{n}\frac{1}{|2\cos \alpha|^{n}}\cdot \left|\frac{2j}{2^{n+1}}-\frac{j-1}{2^{n}}\right|= \\
		\\
		\frac{2}{|2\cos \alpha|^{n+1}}+\frac{2}{|2\cos \alpha|^{n}}\leq \frac{4}{|2\cos \alpha|^{n}}. \\
	\end{array}$
	
	\medskip 
	
	Let $N>1$ a nonnegative integer and observe that
	\begin{center} 
		$\displaystyle \sum_{i=N}^{+\infty} \frac{4}{|2\cos \alpha|^{i}}=\frac{4}{|2\cos\alpha|^{N-1}(2\cos \alpha -1)},$
	\end{center} 
	which converges to zero, since $2\cos \alpha>1$.
	
	For $\varepsilon>0$, let $N>1$ such that $\frac{4}{|2\cos\alpha|^{N-1}(2\cos \alpha -1)}<\varepsilon$. Hence, for all $m>n>N$, we have
	\begin{center} 
		$|x_n(t)-x_m(t)|\leq \displaystyle\sum_{i=n}^{m-1}|x_i(t)-x_{i+1}(t)|\leq \sum_{i=n}^{m-1} \frac{4}{|2\cos \alpha|^{i}}<\varepsilon.$
	\end{center} 
	
	The proof for $t\in \left[\frac{2j-2}{2^{n+1}}, \frac{2j-1}{2^{n+1}}\right]$ is similar. Hence, $(x_n(t))_n$ is a Cauchy sequence, for all $t\in[0,1]$ and therefore, $(x_n)_n$ is a Cauchy sequence in $\mathcal{C}[0,1]$.
	
	In the same way, we prove that  $(y_n)_n$ is a Cauchy sequence in $\mathcal{C}[0,1]$ and we are done.
\end{proof}
Therefore, we define the coordinate functions  $x_\theta,y_\theta:[0,1]\longrightarrow \mathbb{R}$ as
$$x_\theta=\lim_{n\to\infty}x_{\theta,n} \;\;\; \textrm{ and }\;\;\;  y_\theta=\lim_{n\to\infty}y_{\theta,n} .$$ where the existence of these limits is guaranteed by Lemma \ref{cauchy}. 

For graphical representations of the sets $X_\theta$ and $Y_\theta$ corresponding to different angle parameters, see Figures \ref{cord}, \ref{cord4pisobre9}, \ref{cord25pisobre18} and \ref{cord5pisobre6}.

\bigskip 

\begin{figure}[!h]
	\centering
	\includegraphics[scale=0.39]{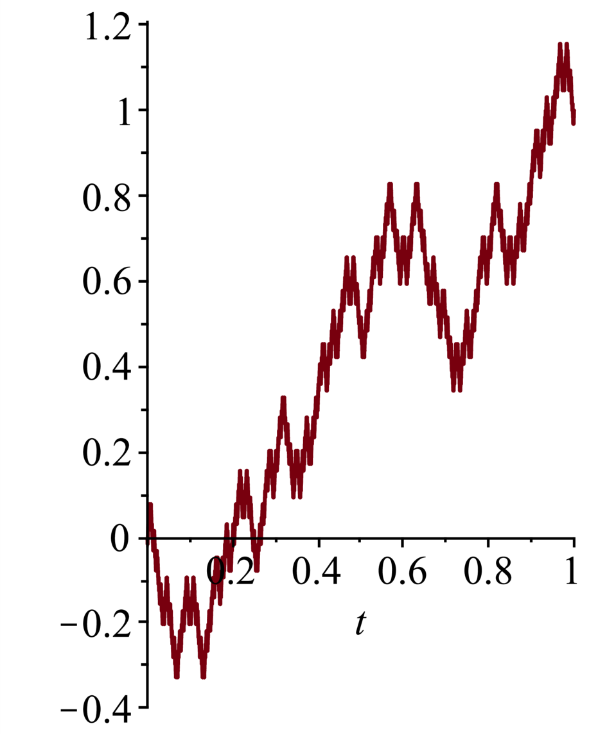} \hfill
	\includegraphics[scale=0.37]{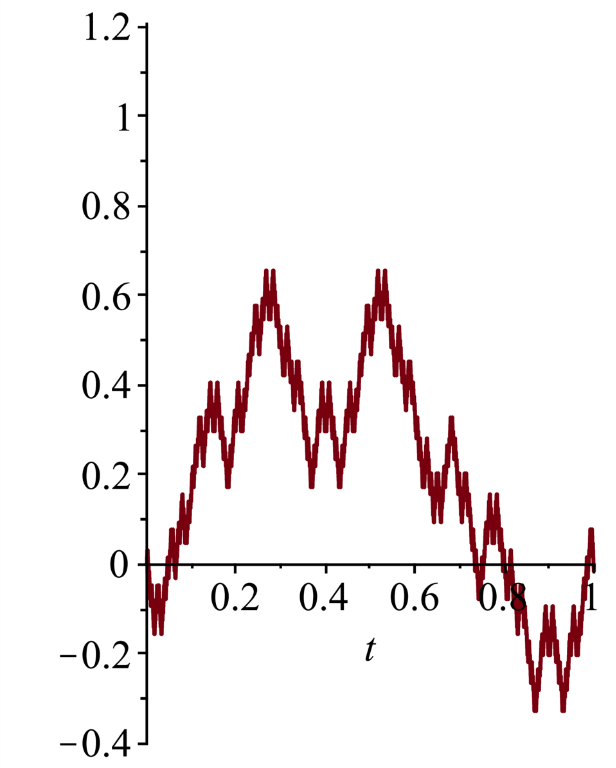}
	\caption{Graphs of $x_{\frac{\pi}{2}}$ and $y_{\frac{\pi}{2}}$.} \label{cord}
\end{figure}

\begin{figure}[!h]
	\centering
	\includegraphics[scale=0.38]{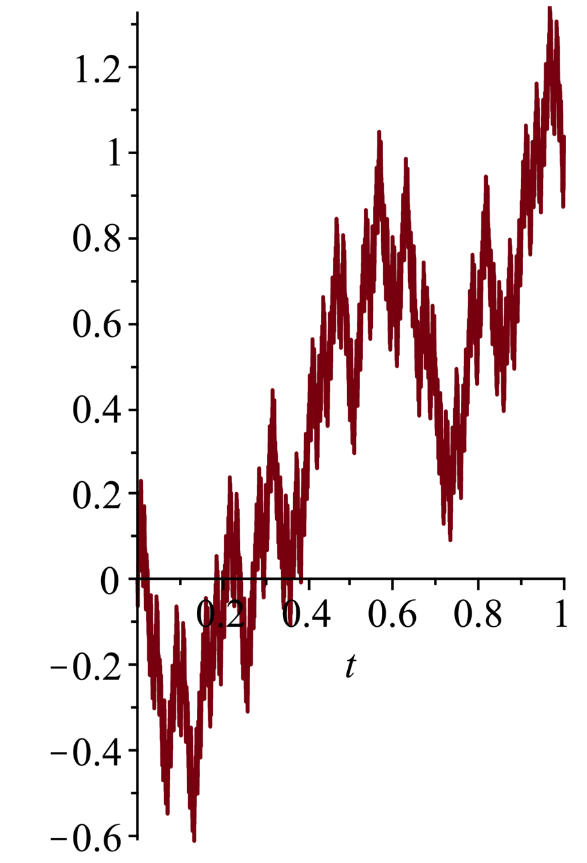}  \hfill
	\includegraphics[scale=0.37]{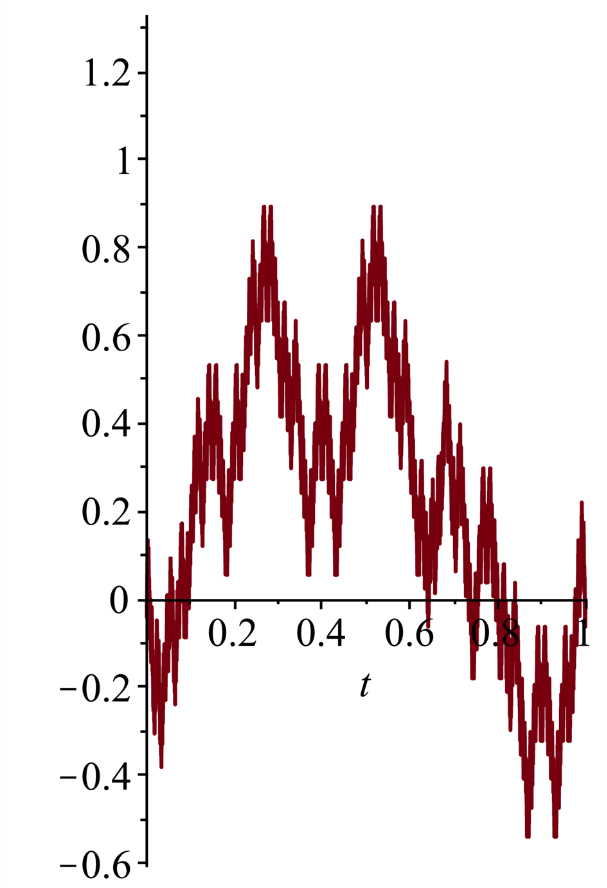}
	\caption{Graphs of $x_{\frac{4\pi}{9}}$ and $y_{\frac{4\pi}{9}}$.} \label{cord4pisobre9}
\end{figure}

\begin{figure}[!h]
	\centering
	\includegraphics[scale=0.39]{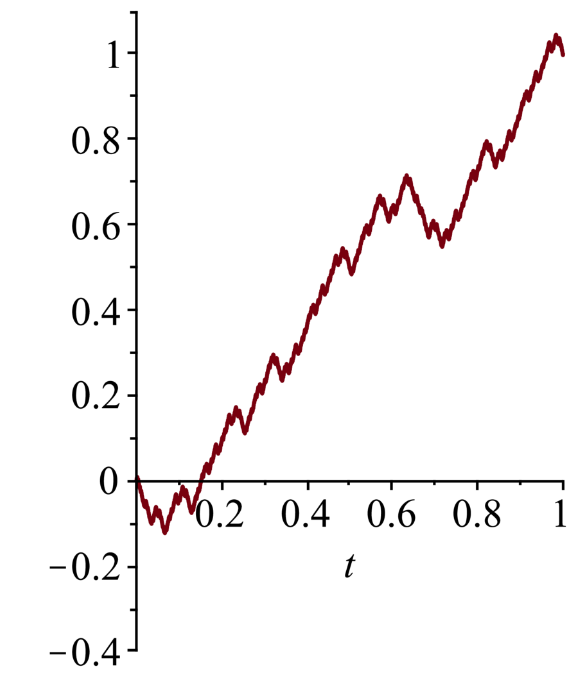} \hfill
	\includegraphics[scale=0.37]{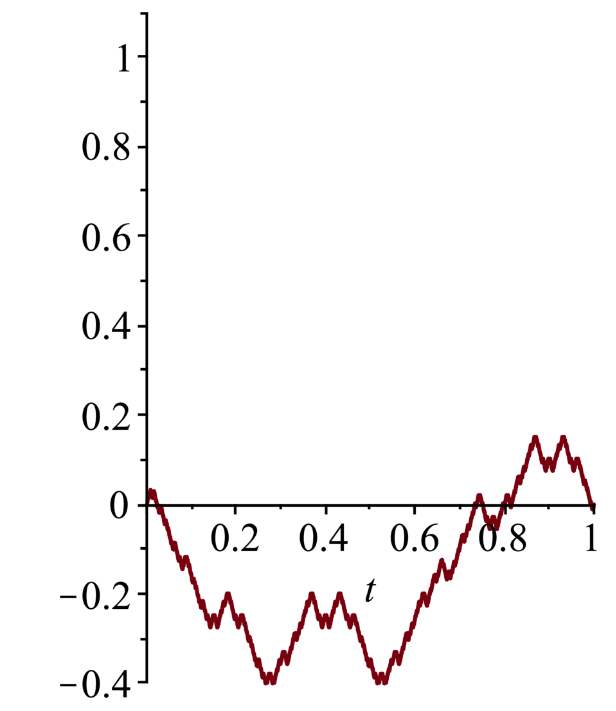}
	\caption{Graphs of $x_{\frac{25\pi}{18}}$ and $y_{\frac{25\pi}{18}}$.} \label{cord25pisobre18}
\end{figure}

\begin{figure}[!h]
	\centering
	\includegraphics[scale=0.37]{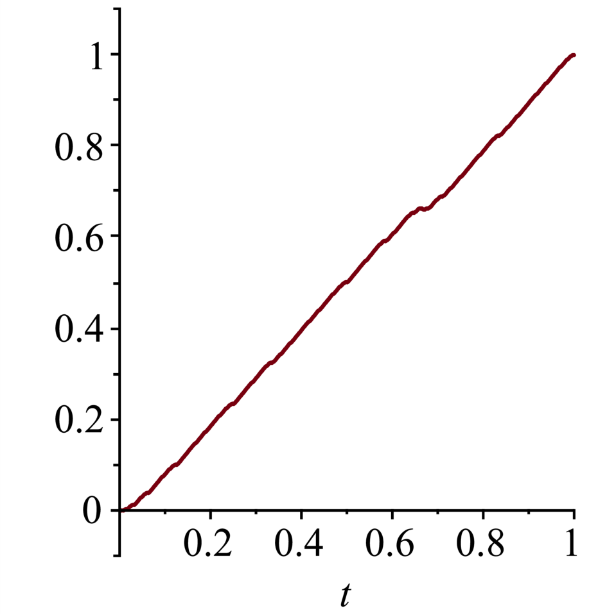}  \hfill
	\includegraphics[scale=0.37]{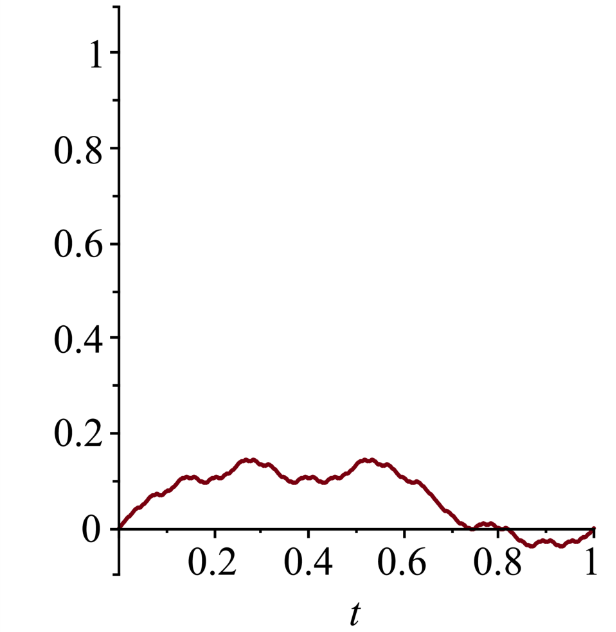}
	\caption{Graphs of $x_{\frac{5\pi}{6}}$ and $y_{\frac{5\pi}{6}}$.} \label{cord5pisobre6}
\end{figure}

Consider $\frac{\pi}{3}<\theta< \frac{5\pi}{3}$, $\alpha=\frac{\pi-\theta }{2}$ and $(x_\theta(t),y_\theta(t))$ the parametrization of the dragon curve $\mathcal{D}_\theta$ associated to the angle $\theta$. The main result of this paper is the following:

\begin{theorem} \label{dimcrorolary}
	We have that the box-counting dimensions of $X_\theta$ and $Y_\theta$ are
	$$\dim_B X_\theta=  \dim_B Y_\theta=  1-\dfrac{\log  (\cos\alpha)}{\log 2}.$$
\end{theorem}

\begin{remark}
	For proving that  $\dim_B X_\theta=  \dim_B Y_\theta=  1-\frac{\log  (\cos\alpha)}{\log 2}$, we will use the following definition of  the box-counting dimension $$\dim_B \mathcal{Z}=\lim_{\delta\to \infty}\dfrac{\log N_\delta (\mathcal{Z})}{- \log \delta}$$
	where $N_\delta (\mathcal{Z})$ is the number of $\delta$-mesh cubes (Figure \ref{figboxcover} is a representation of $\frac{1}{2^k}$-mesh cubes covering the image of $\left[\frac{j-1}{2^k}, \frac{j}{2^k}\right]$ by an affine function) that intersect the set $\mathcal{Z}$ (the definition of box-counting dimension can be check in \cite{f1}).
	
	\begin{figure}[!h]
		\centering
		\includegraphics[scale=0.22]{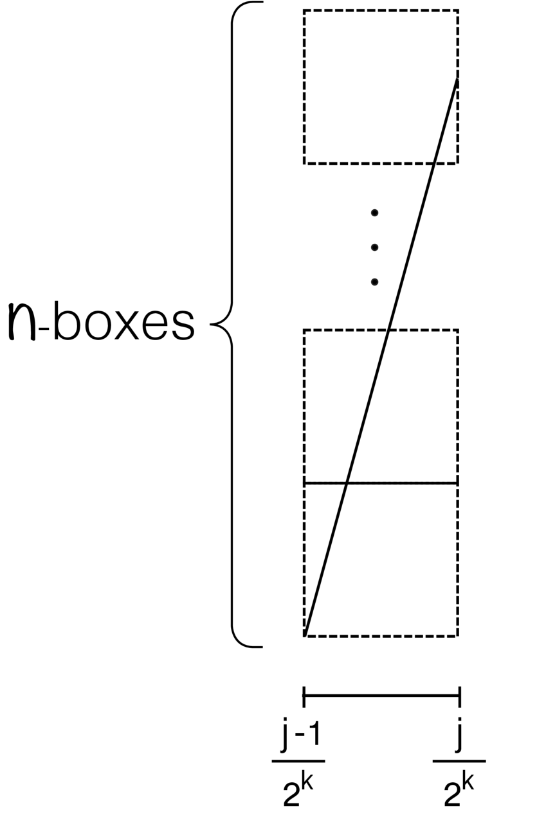}  \;\;\;\;
		\caption{$\frac{1}{2^k}$-mesh cubes covering the image of $\left[\frac{j-1}{2^k}, \frac{j}{2^k}\right]$.}
		\label{figboxcover}
	\end{figure} 
	
	These studies are motivated by work \cite{AK}, were the authors consider the parametrization $\{L(t)=x(t)+i y(t): 0\leq t\leq 1\}$ of Lévy dragon curve, satisfying
	\begin{center} 
		$\left(x\left(\frac{t+k}{4}\right),y\left(\frac{t+k}{4}\right)\right)=\left\{\begin{array}{ll}
			(0,0)+ \frac{1}{2}(-y(t),x(t)), & \textrm{if } k=0 ; \\
			& \\
			(0,\frac{1}{2})+ \frac{1}{2}(x(t),y(t)), & \textrm{if } k=1 ; \\
			& \\
			(\frac{1}{2},\frac{1}{2})+ \frac{1}{2}(x(t),y(t)), & \textrm{if } k=2 ; \\
			& \\
			(1,\frac{1}{2})+ \frac{1}{2}(y(t),-x(t)), & \textrm{if } k=3. \\
		\end{array}  \right.$
	\end{center} 
	and they proved that the Hausdorff and box-counting dimension of $\{(t,x(t)): 0\leq t \leq 1\}$ and $\{(t,y(t)): 0\leq t \leq 1\}$ are equal to $\frac{3}{2}$.
	
	In the same way, it is possible to prove that the Hausdorff and box-counting dimension of $X_{\pi/2}$ and $Y_{\pi/2}$ are equal to $\frac{3}{2}$ from its parametrization $\{x(t)+i y(t): 0\leq t\leq 1\}$, satisfying
	\begin{center} 
		$\left(x\left(\frac{t+k}{4}\right),y\left(\frac{t+k}{4}\right)\right)=\left\{\begin{array}{ll}
			(0,0)+ \frac{1}{2}(-y(t),x(t)), & \textrm{if } k=0 ; \\
			& \\
			(\frac{1}{2},\frac{1}{2})+ \frac{1}{2}(-x(t),-y(t)), & \textrm{if } k=1 ; \\
			& \\
			(1,0)+ \frac{1}{2}(-x(t),-y(t)), & \textrm{if } k=2 ; \\
			& \\
			(\frac{1}{2},\frac{1}{2})+ \frac{1}{2}(y(t),-x(t)), & \textrm{if } k=3. \\
		\end{array}  \right.$
	\end{center} 
	The key of the proof is that, by this parametrization property,  $X_{\pi/2}$ and $Y_{\pi/2}$ each consist of two affine contracted images of  $X_{\pi/2}$ and two of $Y_{\pi/2}$.
	
	On the other hand, in the general case for $\frac{\pi}{3}<\theta <\frac{5 \pi }{3} $, it is unable to guarantee this parametrization property.
\end{remark}

\begin{remark} For proving Theorem \ref{dimcrorolary}, firstly we suppose $\alpha=\frac{p}{q}2\pi$ (rational angle), with $\frac{p}{q}\in\mathbb{Q}^+$ irreducible, that is $0< \frac{p}{q}<\frac{1}{6}$ and fix $\theta=\pi-2\alpha $. From this, we exhibit a number of $\frac{1}{2^k}$-mesh cube that cover $X_{\theta}$  and $Y_{\theta}$, for all $n\geq 0$. Also, we exhibit a number of $\frac{1}{2^k}$-mesh that can not cover $X_\theta$ and $Y_\theta$, for all $n\geq 0$. Hence, we give the proof of Theorem \ref{dimcrorolary} by rational angles and we extend this result for irrational angles using Lemma \ref{allangles}.
\end{remark}

\section{Auxiliary Lemmas} \label{lemmas}

\begin{lemma} \label{fixedpoint}
	For each nonnegative integer $k\geq 1$, we have that 
	
	\begin{center} 
		$x_{k+l}\left(\frac{j}{2^k}\right)=x_k\left(\frac{j}{2^k}\right) \textrm{ and } y_{k+l}\left(\frac{j}{2^k}\right)=y_k\left(\frac{j}{2^k}\right)$
	\end{center} 
	for all $l\geq 1$ and $1\leq  j \leq 2^k$. Furthermore, $$x_k(0)=y_k(0)=y_k(1)=0 \textrm{ and } x_k(1)=1.$$
\end{lemma}

\begin{proof} Observe that
	\begin{center} 
		$x_k\left(\frac{j}{2^{k}}\right)=f_{k,j}+\ldots+f_{k,1}$
	\end{center} 
	and
	\begin{center} 
		$x_{k+1}\left(\frac{j}{2^{k}}\right)=\left(\frac{2j}{2^{k+1}}\right)=f_{k+1,2j}+f_{k+1,2j-1}+\ldots+f_{k+1,2}+f_{k+1,1}.$
	\end{center} 
	
	On the other hand, for each $i=1,\ldots,j$, we have by construction
	\begin{center} 
		$f_{k+1,2i}+f_{k+1,2i-1}=\frac{\cos (b_{k+1,2i})}{(2\cos\alpha)^{k+1}}+\frac{\cos (b_{k+1,2i-1})}{(2\cos\alpha)^{k+1}}$
	\end{center} 
	\begin{center} 
		$ = \frac{\cos (b_{k,i}+\alpha)}{(2\cos\alpha)^{k+1}}+\frac{\cos (b_{k,i}-\alpha)}{(2\cos\alpha)^{k+1}}=\frac{\cos (b_{k,i})}{(2\cos\alpha)^{k}}=f_{k,i}.$
	\end{center} 
	Analogously, we prove that $y_{k+1}\left(\frac{j}{2^k}\right)=y_k\left(\frac{j}{2^k}\right)$ and by induction we are done.
\end{proof}

\begin{lemma} \label{cover}
We have that $X_\theta$ and $Y_\theta$ are covered by
\begin{center} 
$2^k \left\lfloor \frac{2^{k+1}}{|2\cos \alpha|^{k-1}(|2\cos\alpha|-1)}\right\rfloor +2^k$
\end{center} 
$\frac{1}{2^k}$-mesh cubes, respectively, for all $k\geq 1$.
\end{lemma}
\begin{proof}
Let $k\geq 0$ and observe that the image of $\left[\frac{i-1}{2^k},\frac{i}{2^k}\right]$ by $x_k$ has height length equals to $\left|\frac{\cos b_{k,i}}{(2\cos \alpha)^k}\right|$. Thus, from Lemma \ref{fixedpoint}, it is not hard do verify that the image of $\left[\frac{i-1}{2^k},\frac{i}{2^k}\right]$ by $x_{k+j}$ has height length at most equals to
\begin{center} 
	$\displaystyle \frac{1}{|2\cos \alpha|^k}+ 2\sum_{l=k+1}^{k+j}  \frac{1}{|2\cos \alpha|^l}$, for each $j\geq 1$.
\end{center} 
Hence,  the image of $\left[\frac{i-1}{2^k},\frac{i}{2^k}\right]$ by $x_\theta$ has height length at most equals to
\begin{center} 
$\displaystyle \sum_{l=k}^{+\infty}  \frac{2}{|2\cos \alpha|^l}=\frac{2}{|2\cos \alpha|^{k-1}(|2\cos\alpha|-1)}$
\end{center} 
and so $X_\theta$ is covered by 
\begin{center} 
	$2^k \left\lfloor \frac{2^{k+1}}{|2\cos \alpha|^{k-1}(|2\cos\alpha|-1)}\right\rfloor +2^k$
\end{center} 
$\frac{1}{2^k}$-mesh cubes.

The proof for coverage of $Y_\theta $ can be done in the same way.
\end{proof}

\begin{definition}
	Let $k,l\geq 0$ be two nonnegative integers. We define a relation $\sim$ over $\mathcal{D}_k$ by the following. We say that $$\mathcal{D}_k\sim \{(\eta_{k,0},\theta_{k,0}),\ldots,(\eta_{k,l},\theta_{k,l})\}$$ if $\mathcal{D}_k$ is composed by $\eta_{k,i}$ line segments with angle $\theta_{k,i}$, for $i=0,\ldots,l$.
\end{definition}

\begin{remark}
	For example, we have that $$\mathcal{D}_k\sim \{(1,b_{k,1}),(1,b_{k,2}),\ldots,(1,b_{k,2^k})\}.$$
	However, the representation of this relation is not unique, since we can have $b_{k,i}=b_{k,j}$ with $i\neq j$. In this case, we can represent $(1,b_{k,i})(1,b_{k,j})$ by $(2,b_{k,i})$. The idea for proving the main result is to consider a simplest representation for the relation $\sim$ as the following:
	$$\mathcal{D}_k\sim \{(\eta_{k,0},\theta_{k,0}),\ldots,(\eta_{k,l},\theta_{k,l})\}$$
	where $\theta_{k,i}\neq \theta_{k,j}$ if $i\neq j$ and, it is possible if $k=nq-1$. Furthermore, from Lemma \ref{cauchy}, in reason to estimate the box-counting dimensions of $X_\theta$ and $Y_\theta$, it is sufficient to consider the subsequence $(x_{qn-1})_n$ and $(y_{qn-1})_n$ of coordinate functions. Also, since $\alpha=\frac{p}{q}2\pi$, it follows that $b_{k,i}\in\{0,\alpha,2\alpha,\ldots,(q-1)\alpha\}$, for all $k\geq 0$ and $i\in\{1,\ldots,2^k\}$.
\end{remark}

\begin{lemma} \label{pascal}
For all $k\geq 0$, $\mathcal{D}_k\sim  \{(\eta_{k,0},\theta_{k,0}),\ldots ,(\eta_{k,k},\theta_{k,k})\}$, where $\eta_{k,0},\ldots, \eta_{k,k}$ is the $k$-nth row of Pascal's triangle  and $\theta_{k,i}= k\alpha -2i\alpha$, for $i=0,\ldots,k $, that is
\begin{center}
	$\eta_{k,i}=\displaystyle \left(\begin{array}{c}
		k \\ i 
	\end{array}\right)=\dfrac{k!}{i!(k-i)!}$
\end{center}
 for $i=0,\ldots,k$, where $n!=n\cdot (n-1) \cdot \ldots \cdot 1$.
\end{lemma}
\begin{proof}  We know that $\mathcal{D}_k$ is composed by a sequence of consecutive line segments $s_{k,i}$ which each one has angle $b_{k,i}$, for 
	$i=1,\ldots,2^k$. For example,
	
	$$\begin{array}{lll}
		\mathcal{D}_0: & b_{0,1} & = 0, \\
		\mathcal{D}_1: & b_{1,1},b_{1,2} & = \alpha, -\alpha , \\
		\mathcal{D}_2: & b_{2,1},b_{2,2},b_{2,3},b_{2,4} & =  2\alpha, 0, -2\alpha, 0,
	\end{array}$$
	which implies
	$$\begin{array}{l}
		\mathcal{D}_0 \sim  \{(1,0)\}, \\
		\mathcal{D}_1 \sim  \{(1,\alpha), (1,-\alpha)\}, \\
		\mathcal{D}_2 \sim  \{(1,2\alpha), (2,0), (1,-2\alpha)\}.
	\end{array}$$
	
	\noindent Hence, observe that for $j=0,1,2$, $\mathcal{D}_j\sim (\eta_{j,0},\theta_{j,0})\ldots (\eta_{j,j},\theta_{j,j})$, where $\eta_{j,0},\ldots, \eta_{j,j}$ is the $j$-nth row of Pascal's triangle  and $\theta_{j,i}= j\alpha -2i\alpha$, for $i=0,\ldots,j $.
	
Now, by induction, let $l\geq 2$ and suppose that $$\mathcal{D}_l\sim \{(\eta_{l,0},\theta_{l,0}),\ldots ,(\eta_{l,l},\theta_{l,l})\},$$ where $\eta_{l,0},\ldots, \eta_{l,l}$ is the $l$-nth row of Pascal's triangle and $\theta_{l,i}= l\alpha -2i\alpha$, for $i=0,\ldots,l $.

Hence, since $\theta_{l,i}= l\alpha -2i\alpha$ and $$\theta_{l,i-1}-\alpha=l\alpha -2(i-1)\alpha -\alpha= l\alpha -2i\alpha +\alpha = \theta_{l,i}+\alpha,$$ it follows by construction of $\mathcal{D}_{l+1}$ from $\mathcal{D}_{l}$, that
$$\mathcal{D}_{l+1}\sim \{(\eta_{l,0},\theta_{l,0}+\alpha), (\eta_{l,0},\theta_{l,0}-\alpha),\ldots ,(\eta_{l,l},\theta_{l,l}+\alpha),(\eta_{l,l},\theta_{l,l}-\alpha)\}=$$
$$\begin{array}{l}
	\{(\eta_{l,0},\theta_{l,0}+\alpha), \\
	\hspace{.2cm} (\eta_{l,0},\theta_{l,0}-\alpha),(\eta_{l,1},\theta_{l,1}+\alpha), \\
	\hspace{.2cm} (\eta_{l,1},\theta_{l,1}-\alpha),(\eta_{l,2},\theta_{l,2}+\alpha), \\
	\hspace{3.2cm} \vdots \\
	\hspace{.2cm} (\eta_{l,i-1},\theta_{l,i-1}-\alpha),(\eta_{l,i},\theta_{l,i}+\alpha), \\
	\hspace{3.2cm} \vdots \\
	\hspace{.2cm} (\eta_{l,l-1},\theta_{l,l-1}-\alpha),(\eta_{l,l},\theta_{l,l}+\alpha), \\
	\hspace{.2cm} (\eta_{l,l},\theta_{l,l}-\alpha)\}
\end{array} \;\;\;\; = \;\;\;\;  \begin{array}{l}
	\{(\eta_{l+1,0},\theta_{l+1,0}), \\
	\hspace{.2cm} (\eta_{l+1,1},\theta_{l+1,1}), \\
	\hspace{.2cm} (\eta_{l+1,2},\theta_{l+1,2}), \\
	\hspace{1.3cm} \vdots \\
	\hspace{.2cm} (\eta_{l+1,i},\theta_{l+1,i}), \\
	\hspace{1.3cm} \vdots \\
	\hspace{.2cm} (\eta_{l+1,l},\theta_{l+1,l}), \\
	\hspace{.2cm} (\eta_{l+1,l+1},\theta_{l+1,l+1})\},
\end{array}$$
where
\begin{itemize}
	\item $\eta_{l+1,0}=\eta_{l,0}=\left(\begin{array}{c} l \\ 0 \end{array}\right)= \left(\begin{array}{c} l+1 \\ 0 \end{array}\right)$,
	
	\item $\eta_{l+1,i}=\left(\begin{array}{c} l \\ i-1
	\end{array}\right)+\left(\begin{array}{c} l \\ i
	\end{array}\right)=\eta_{l,i-1}+\eta_{l,i} =\left(\begin{array}{c} l+1 \\ i \end{array}\right)$, \newline  for $i=1,\ldots,l$,
	
	\item $\eta_{l+1,l+1}=\eta_{l,l}=\left(\begin{array}{c} l \\ l \end{array}\right)= \left(\begin{array}{c} l+1 \\ l+1 \end{array}\right)$,
\end{itemize}
and
\begin{itemize}
	\item $\theta_{l+1,i}=\theta_{l,i}+\alpha=l\alpha-2i\alpha +\alpha =(l+1)\alpha-2i\alpha$, for $i=0,\ldots, l$,
	
	\item $\theta_{l+1,l+1}=\theta_{l,l}-\alpha=l\alpha-2l\alpha -\alpha =(l+1)\alpha-2(l+1)\alpha$.
\end{itemize}
With this, we finish the prove of lemma.
\end{proof}

\begin{lemma} \label{noncover} We have that $X_\theta$ and $Y_\theta$ can not be covered by
	\begin{center}
		$2^{k-1}\left\lfloor\frac{\lambda_1 }{2(\cos\alpha)^{k}}\right\rfloor$ \;\;\; and \;\;\; $2^{k-1}\left\lfloor\frac{\lambda_2 }{2(\cos\alpha)^{k}}\right\rfloor$
	\end{center}
$\frac{1}{2^{k}}$-mesh cubes, respectively, for all $k$ sufficiently large where $$\displaystyle \lambda_1 =\min_{0\leq j< q} \{|\cos j\alpha| : \cos j\alpha \neq 0\}$$ and $$\displaystyle \lambda_2 =\min_{0\leq j< q} \{|\sin j\alpha| : \sin j\alpha \neq 0 \}.$$
\end{lemma}
\begin{proof} Firstly, suppose $b_{k,j}\neq \frac{\pi}{2} , \frac{3\pi}{2}$ and $b_{k,j}\neq 0 , \pi$. Thus, the image of $\left[\frac{j-1}{2^k},\frac{j}{2^k}\right]$ by $x_k$ is a hypotenuse of a right triangle with base measuring $\frac{1}{2^k}$ and height measuring $f_{k,j}$. Also, the image of $\left[\frac{j-1}{2^k},\frac{j}{2^k}\right]$ by $y_k$ is a hypotenuse of a right triangle with base measuring $\frac{1}{2^k}$ and height measuring $g_{k,j}$.
Since, the least number of boxes necessary to cover a hypotenuse of a right triangle is when the diagonal boxes is contained in the respectively hypotenuse, as in Figure \ref{figbox}, and we have that
\begin{center} 	
	$\left\lfloor\dfrac{\sqrt{\frac{1}{2^{2k}}+\frac{(\cos b_{k,j})^2}{(2\cos \alpha)^{2k}}}}{\frac{1}{2^{k}}}\right\rfloor= 
	\left\lfloor \sqrt{1+\frac{(\cos b_{k,j})^2}{(\cos \alpha)^{2k}}}\right\rfloor>\left\lfloor\frac{1}{2} \left|\frac{\cos b_{k,j}}{(\cos \alpha)^{k}}\right|\right \rfloor $
\end{center} 
and
\begin{center} 	
	$\left\lfloor\dfrac{\sqrt{\frac{1}{2^{2k}}+\frac{(\sin b_{k,j})^2}{(2\cos \alpha)^{2k}}}}{\frac{1}{2^{k}}}\right\rfloor= 
	\left\lfloor \sqrt{1+\frac{(\sin b_{k,j})^2}{(\cos \alpha)^{2k}}}\right\rfloor>\left\lfloor\frac{1}{2} \left|\frac{\sin b_{k,j}}{(\cos \alpha)^{k}}\right|\right \rfloor $
\end{center} 
it follows that 
\begin{center} 	
$\left\lfloor\frac{1}{2} \left|\frac{\cos b_{k,j}}{(\cos \alpha)^{k}}\right|\right \rfloor $ \;\;\; and \;\;\; 
$\left\lfloor\frac{1}{2} \left|\frac{\sin b_{k,j}}{(\cos \alpha)^{k}}\right|\right \rfloor $
\end{center} 
boxes of side $\frac{1}{2^k}$ can not cover the image of $\left[\frac{j-1}{2^k},\frac{j}{2^k}\right]$ by $x_k$ and $y_k$ respectively.

\begin{figure}[!h]
	\centering
	\includegraphics[scale=0.22]{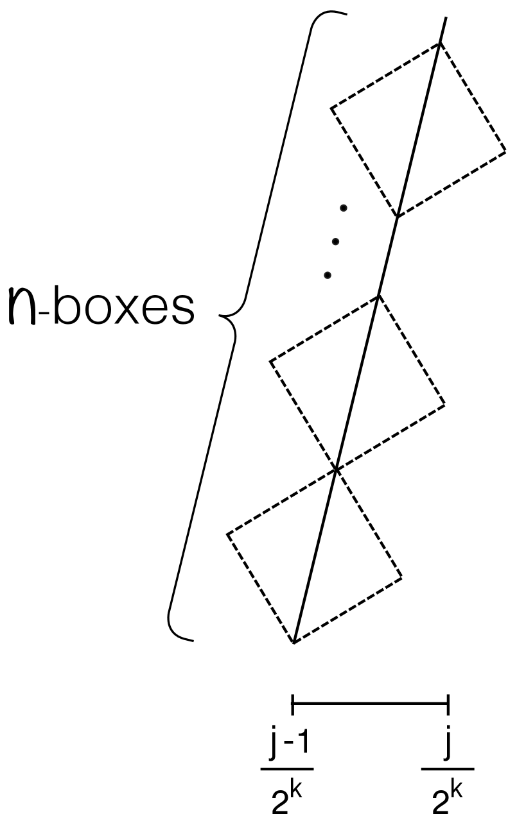}
	\caption{Boxes with side $\frac{1}{2^k}$ covering the image of $\left[\frac{j-1}{2^k}, \frac{j}{2^k}\right]$ by $x_k$ and $y_k$.}
	\label{figbox}
\end{figure} 
	
Hence, from construction of $\mathcal{D}_{k+1}$ from $\mathcal{D}_k$ and by Lemma \ref{fixedpoint}, it follows that 
	\begin{center}
	$\left\lfloor\frac{\lambda_1 }{2(\cos\alpha)^k}\right\rfloor$ \;\;\; and \;\;\; $\left\lfloor\frac{\lambda_2 }{2(\cos\alpha)^k}\right\rfloor$
\end{center}
boxes of side $\frac{1}{2^k}$ can not cover the image of $\left[\frac{j-1}{2^k},\frac{j}{2^k}\right]$ by $x_n$ and $y_n$ respectively, for all $n\geq k$. Therefore, it can not cover the image of $\left[\frac{j-1}{2^k},\frac{j}{2^k}\right]$ by $x_\theta$ and $y_\theta$.

On the other hand, if $k=2l$ is an even number, then from Lemma \ref{pascal} we have that
$$\mathcal{D}_k\sim \{(\eta_{k,0},\theta_{k,0}),\ldots,(\eta_{k,k},\theta_{k,k})\}$$
where $\theta_{k,i}=k\alpha-2i\alpha=2(l-i)\alpha$, for $i=0,\ldots k$, that is

$$ \begin{array}{ccccc}
\theta_{k,0} & = &  2l\alpha & = & -\theta_{k,k} ,\\
\theta_{k,1} & = &  2(l-1)\alpha & = & -\theta_{k,k-1} , \\
\vdots & \vdots   & \vdots  & \vdots & \vdots  \\
\theta_{k,i} & = & 2(l-i)\alpha & = & -\theta_{k,k-i} ,  \\
\vdots & \vdots  & \vdots  & \vdots & \vdots \\
\theta_{k,\frac{k}{2}-1} & = &  2\alpha & = & -\theta_{k,\frac{k}{2}+1}. \\
  \theta_{k,\frac{k}{2}} & = &  0, & &  \\
\end{array} $$

Let  $j = \min \{1 \leq i \leq l : 2i\alpha = \frac{\pi}{2} \textrm{ or } -\frac{\pi}{2}\}$, if it exists. In this case, since $q \alpha =0$, there exists a smallest nonnegative integer $0<m< q$ such that $|\theta_{k,i}|=\frac{\pi}{2}$ if, and only if, $Q_{k,i}=\pm 2(j+rm)\alpha$, that is, $i=l\mp (j+rm)$, with $r=0,\ldots, \left\lfloor \frac{k-j}{m}\right\rfloor$.

Furthermore, since by periodicity of $\alpha$, that is, $q>6$ is the smallest nonnegative integer such that $q\alpha=0$, it follows that $2j$ is the smallest nonnegative integer such that $2j\alpha= \pi$ or $-\pi$ and we have $3j\alpha=\mp \frac{\pi}{2}$, with $n\alpha\neq \pm\frac{\pi}{2}$, for $j<n<3j$. Thus, $m=2j\geq 2$.

Let $r=\left\lfloor \frac{k-j}{m}\right\rfloor$. Therefore, there exists at most
$$\Lambda:=\displaystyle  \sum_{i=0}^{r}\left(\begin{array}{c}
	k \\ l-(j+(r-i)m) 
\end{array}\right) +
\sum_{i=0}^{r}\left(\begin{array}{c}
	k \\ l+(j+(r-i)m) 
\end{array}\right) $$
angles with zero cosine. Since $m\geq 2$, it follows that $\Lambda \leq 2^{k-1}$.

Hence, we have that
\begin{center}
$\left\lfloor\frac{\lambda_1 }{2(\cos\alpha)^k}\right\rfloor(2^k-2^{k-1})=\left\lfloor\frac{\lambda_1 }{2(\cos\alpha)^k}\right\rfloor2^{k-1}$
\end{center}
boxes of side $\frac{1}{2^k}$ (and therefore $\frac{1}{2^k}$-mesh cubes) can not cover $X_\theta$.

The proof for $k$ odd and for $Y_\theta$ can be done in the same way.
\end{proof}

\begin{lemma} \label{allangles}
Let $(\theta_n)_n$ be a sequence of rational angles in $\left]\frac{\pi}{3}, \pi\right[ $ converging to a irrational angle $\theta\in \left]\frac{\pi}{3}, \pi\right[ $. Then, we have that $ x_{\theta_n}$ and $y_{\theta_n}$ converges uniformly to $x_{\theta}$ and $ y_\theta$, respectively. \end{lemma}
\begin{proof}
Let us consider

\begin{itemize}
	\item $x_{\theta_i,k}$ the coordinate function of $\mathcal{D}_{\theta_i,k}$ associated to the angle $\theta_i$;

	\item $x_{\theta_i}$ the coordinate function of the dragon curve $\mathcal{D}_{\theta_i}$ associated to the angle $\theta_i$;
	
	\item $x_{\theta,k}$ the coordinate function of $\mathcal{D}_{\theta,k}$ associated to the angle $\theta$;
	
	\item $x_{\theta}$ the coordinate function of the dragon curve $\mathcal{D}_{\theta}$ associated to the angle $\theta$.
\end{itemize}

Hence, from Lemma \ref{cauchy}, we have that $x_{\theta_i,k}$ converges uniformly to $x_{\theta_i}$, when $k$ goes to infinity, for all nonnegative integer $i$.

In the same way of proof of Lemma \ref{cauchy}, we have that $x_{\theta,k}$ converges uniformly to $x_{\theta}$, when $k$ goes to infinity.

On the other hand, it is not hard to prove that $x_{\theta_i,k}$ converges uniformly to $x_{\theta,k}$, when $i$ goes to infinity, for all nonnegative integer $k$.

In the other words, we have

\begin{equation*}
	\xymatrix@R+2em@C+2em{
		x_{\theta_i,k} \ar[r]^-{ \rotatebox[origin=c]{0}{ \tiny $k\to +\infty$} }_-{\rotatebox[origin=c]{0}{ \tiny uniformly}} \ar[d]^-{ \tiny \begin{turn}{-90}
$i\to +\infty$ \end{turn} }_-{ \rotatebox[origin=c]{-90}{ \tiny uniformly}  } & x_{\theta_i}  \\
		x_{\theta,k} \ar[r]^-{\rotatebox[origin=c]{0}{ \tiny $k\to +\infty$} }_-{\rotatebox[origin=c]{0}{ \rotatebox[origin=c]{0}{ \tiny uniformly}}} & x_\theta
	}
\end{equation*}

\medskip 

Let $\alpha_i= \frac{\pi -\theta_i}{2}$, $\alpha= \frac{\pi -\theta}{2}$ and consider $N_1$ sufficiently  large such that $\displaystyle \inf_{i\geq N_1}\{\alpha_i \} >0$ and ${\displaystyle \sup_{i\geq N_1}\{\alpha_i \}<}\frac{\pi}{3}$. Furthermore, consider 

\begin{center}
$\lambda ' = \displaystyle \inf_{i\geq N_1}\{|2\cos \alpha_i| \} $ and $\lambda = \min \{\lambda' , |2\cos \alpha| \} $.
\end{center} 

Since, from proof of Lemma \ref{cauchy}, we have
\begin{center} 
$ |x_{\theta_i,n} -x_{\theta_i,m}|<\frac{4}{|2\cos\alpha|^{n-1}(2\cos\alpha - 1) }$
\end{center} 
for all non-negative integers $m>n$, it follows that
\begin{center} 
$ |x_{\theta_i,n} -x_{\theta_i}|<\frac{8}{\lambda^{n-1}(\lambda- 1) }$
\end{center} 
for all non-negative integers $i,n\in\mathbb{N}$, with $i\geq N_1$.

By the same way, we have that 
\begin{center} 
$ |x_{\theta,n} -x_{\theta}|<\frac{8}{\lambda^{n-1}(\lambda- 1) }$
\end{center} 
for all non-negative integers $n\in\mathbb{N}$.

Let $\varepsilon> 0$ and fix $k$ sufficiently large such that
\begin{center} 
$\frac{8}{\lambda^{k-1}(\lambda- 1) } <\frac{\varepsilon}{3}$
\end{center} 
and $N> N_1$ such that $\displaystyle |x_{\theta_i,k}-x_{\theta,k}|<\frac{\varepsilon}{3}$, for all $i\geq N$.

Hence,

$$|x_{\theta}-x_{\theta_i}|\leq |x_{\theta}-x_{\theta,k}|+|x_{\theta,k}-x_{\theta_i,k}|+|x_{\theta_i,k}-x_{\theta_i}|<\varepsilon.$$
for all $i\geq N$.

Therefore, $(x_{\theta_i})_i$ converges uniformly to $x_{\theta}$.
\end{proof}

\section{Proof of Theorem \ref{dimcrorolary}} \label{proofmaintheorem}
Firstly, let $\theta$ a rational angle. From Lemma \ref{cover}, we have that

\begin{center} 
	$\dim_B X_\theta \leq \displaystyle \lim_{k\to +\infty} \displaystyle \frac{\log 2^k \left( \frac{2^{k+1}}{|2\cos \alpha|^{k-1}(|2\cos\alpha|-1)} +2\right)}{\log 2^k} =
	\lim_{k\to +\infty} \left[ 1+  \frac{\log\left(\frac{4}{(\cos \alpha)^{k-1}(2\cos\alpha-1)} +2\right)}{k\log 2}\right] \leq \lim_{k\to +\infty} \left[1+  \frac{\log \frac{6}{(\cos \alpha)^{k-1}(2\cos\alpha-1)} }{k\log 2} \right] =  \lim_{k\to +\infty} $$ \left[1+  \frac{\log 6-\log (\cos \alpha)^{k-1}-\log (2\cos\alpha-1)}{k\log 2} \right] = 1- \frac{\log \cos \alpha }{\log 2} . $ 
\end{center}

	On the other hand, from Lemma \ref{noncover}, we have that
\begin{center} 
	$\dim X_\theta \geq \displaystyle\lim_{k\to \infty }\frac{\log  2^{k-1}\left(\frac{\lambda_1 }{2(\cos\alpha)^{k}}-1\right)}{\log 2^k} \geq \lim_{k\to \infty }   $$ \left[\frac{k-1}{k}+\displaystyle \frac{\log \left(\frac{2\lambda_1 }{2(\cos\alpha)^{k}}\right)}{\log 2^k}\right] = \displaystyle \lim_{k\to \infty }   $$ \left[\frac{k-1}{k}+ \frac{\log \lambda_1 - \log (\cos\alpha)^{k}}{\log 2^k}\right] =1- \frac{\log \cos \alpha }{\log 2} .$	
\end{center} 

Therefore, $\dim_B X_\theta = 1-\frac{\log  (\cos\alpha)}{\log 2}$. In analogous way, we prove that  $\dim_B Y_\theta = 1-\frac{\log  (\cos\alpha)}{\log 2}$.
	
Now, let $\alpha$ be an irrational angle. From Lemma \ref{allangles}, there exist a sequence of rational angles $(\alpha_n)_{n\geq 1}$ such that 

\begin{center} 
	$\|x_{\theta_n}-x_{\theta}\|_\infty<\frac{1}{2^{n+1}}$,
\end{center} 
where $\alpha_n=\frac{\pi-\theta_n}{2}$, for all $n\geq 1$. 

For each nonnegative integer $m\in\mathbb{N}$, from Lemma \ref{noncover} we have that $X_{\theta_m}$ can not be covered by
	\begin{center}
	$2^{k-1}\left\lfloor\frac{\lambda_m }{2(\cos\alpha_m)^{k}}\right\rfloor$
\end{center}
$\frac{1}{2^{k}}$-mesh cubes, for all $k$ sufficiently large where
\begin{center} 
 $\displaystyle \lambda_m =\min_{0\leq j< q_m} \{|\cos j\alpha_m| : \cos j\alpha_m \neq 0\}$  and $\alpha_m=\frac{p_m}{q_m}2\pi $.
 \end{center}

Since
\begin{center} 
	$\|x_{\theta_m}-x_{\theta}\|_\infty<\frac{1}{2^{m+1}}$
\end{center} 
from Lemma \ref{fixedpoint} and by the same idea of the proof of Lemma \ref{noncover}, it follows that $X_\theta$ can not be covered by  
\begin{center}
	$2^{k-1}\left(\left\lfloor\frac{\lambda_1 }{2(\cos\alpha_m)^{k}}\right\rfloor -1\right)$
\end{center}
$\frac{1}{2^{k}}$-mesh cubes, for all $k\geq m$ sufficiently large. Hence,
\begin{center} 
	$\dim X_\theta \geq \displaystyle\lim_{k\to \infty }\frac{\log  2^{k-1}\left(\frac{\lambda_1 }{2(\cos\alpha_m)^{k}}-2\right)}{\log 2^k}=1- \frac{\log \cos \alpha_m }{\log 2} $
\end{center}
for all $m\geq 1$ that is,
\begin{center} 
	$\dim X_\theta \geq 1- \frac{\log \cos \alpha }{\log 2} $
\end{center}

On the other hand, from Lemma \ref{cover} we have that $X_{\theta_m}$ is covered by
\begin{center} 
	$2^k \left\lfloor \frac{2^{k+1}}{|2\cos \alpha_m|^{k-1}(|2\cos\alpha_m|-1)}\right\rfloor +2^k$
\end{center} 
$\frac{1}{2^k}$-mesh cubes.

Since
\begin{center} 
	$\|x_{\theta_m}-x_{\theta}\|_\infty<\frac{1}{2^{m+1}}$
\end{center} 
from Lemma \ref{fixedpoint} and by the same idea of the proof of Lemma \ref{cover}, it follows that $X_\theta$ is covered by 
\begin{center} 
	$2^k \left(\left\lfloor \frac{2^{k+1}}{|2\cos \alpha_m|^{k-1}(|2\cos\alpha_m|-1)}\right\rfloor + 1 \right) +2^k$
\end{center} 
for all $k\geq m$. Thus,
\begin{center} 
	$\dim_B X_\theta \leq \displaystyle \lim_{k\to +\infty} \displaystyle \frac{\log 2^k \left( \frac{2^{k+1}}{|2\cos \alpha_m|^{k-1}(|2\cos\alpha_m|-1)} +3\right)}{\log 2^k} =
 1- \frac{\log \cos \alpha_m }{\log 2} $ 
\end{center} 
for all $m\geq 1$ that is,
\begin{center} 
	$\dim X_\theta \leq 1- \frac{\log \cos \alpha }{\log 2} $.
\end{center}

Therefore
\begin{center} 
	$\dim X_\theta = 1- \frac{\log \cos \alpha }{\log 2} $.
\end{center}
and we finish the proof of theorem. \hfill $ \square$

\section{Open problems} \label{conclusion}

In this section, we discuss some problems related to our work.

\begin{problem}
Can we prove that the Haussdorf dimensions of $X_\theta$ and $Y_\theta$ are both equal to $$1-\frac{\log \cos \alpha}{\log 2}$$ respectively?
\end{problem}

Another interesting question concerns self-intersective dragon curves (see \cite{Allo}). For example, it is known that paperfolding curves are self-intersecting for angles $ \frac{\pi}{4} < \alpha < \frac{\pi}{2} $ but non-self-intersecting for $ \alpha = 0 $ and $ \alpha = \frac{\pi}{4} $. Furthermore, in \cite{Alb}, Albers proved that paperfolding curves are non-self-intersecting for angles $ \frac{\pi}{4} < \alpha < 42.437^\circ $. Recently, in \cite{AKW}, the authors showed that paperfolding curves are non-self-intersecting for $ 0 < \alpha < 40.3281^\circ $. The question remains open for angles $ 40.3281^\circ \leq \alpha \leq 42.437^\circ $.

\begin{problem}
What is the relationship between self-intersecting dragon curves and their box-counting dimension?
\end{problem}

Computer simulations suggest that some dragon curves are connected sets with nonempty interiors, while others are connected but contain holes. In some cases, the dragon curve appears to have an empty interior. This leads to the following question:

\begin{problem}
From the box-counting dimension of the coordinate functions of the dragon curve $ \mathcal{D}_{\alpha} $, can we deduce information about the Lebesgue measure of $ \mathcal{D}_{\alpha} $ or the connectedness of $ \mathbb{R}^2 \setminus \mathcal{D}_{\alpha} $?
\end{problem}

\end{document}